\documentclass[center]{aspm}
\articleinfo{}{}{}
\usepackage{mathtools,amssymb,amsthm}
\usepackage{eucal}
\usepackage{cite}
\usepackage[hidelinks]{hyperref}
\ifdefined\pdfpagewidth
  \AtBeginDocument{\pdfpagewidth=210mm\pdfpageheight=297mm}
\fi
\hypersetup{pdftitle={Stability of Kantorovich potentials via heat kernel regularization},
 pdfauthor={Bang-Xian Han and Zhuo-Nan Zhu},
 pdfsubject={A survey of heat kernel regularization and stability of Kantorovich potentials},
 pdfkeywords={optimal transport, Kantorovich potential, heat kernel, RCD space, Heisenberg group}}
\allowdisplaybreaks[1]
\numberwithin{equation}{section}
\theoremstyle{plain}
\newtheorem{proposition}{Proposition}[section]
\newtheorem{lemma}[proposition]{Lemma}
\newtheorem{maintheorem}{Theorem}

\newcommand{\R}{\mathbb R}
\newcommand{\HH}{\mathbb H}
\newcommand{\PP}{\mathcal P}
\newcommand{\mm}{\mathfrak m}
\newcommand{\dist}{\mathrm d}
\newcommand{\distcc}{\mathrm d_{\mathrm{cc}}}
\newcommand{\dd}{\,\mathrm d}
\newcommand{\E}{\mathbb E}
\newcommand{\Var}{\operatorname{Var}}

\newcommand{\Lip}{\operatorname{Lip}}
\newcommand{\diam}{\operatorname{diam}}
\newcommand{\K}{\mathcal K}
\newcommand{\norm}[1]{\lVert #1\rVert}
\newcommand{\nuxt}{\nu_x^{t,\theta}}
\newcommand{\nuhat}{\widehat\nu_x^{t,\theta}}
\title[Stability of Kantorovich potentials]
 {Stability of Kantorovich potentials via heat kernel regularization}
\author[B.-X. Han]{Bang-Xian Han}
\address{School of Mathematics, Shandong University, Jinan 250100, China}
\email{hanbx@sdu.edu.cn}
\author[Z.-N. Zhu]{Zhuo-Nan Zhu}
\address{School of Mathematical Sciences, University of Science and Technology
of China, Hefei 230026, China}
\email{zhuonanzhu@mail.ustc.edu.cn}
\thanks{This work is supported by the Young Scientist Programs of the
Ministry of Science \& Technology of China (2021YFA1000900, 2021YFA1002200),
the National Natural Science Foundation of China (12201596, 12671067),
the Shandong Provincial Natural Science Foundation (ZR2025QB05), and
the Taishan Scholars Program of Shandong Province (tsqn202408059).}
\rcvdate{}
\rvsdate{}
\makeatletter
\@ifundefined{subjclassname@2020}{%
  \@namedef{subjclassname@2020}{2020 Mathematics Subject Classification}}{}
\makeatother

\begin{document}
\begin{abstract}
We survey heat kernel regularization as a method for quantitative
stability of Kantorovich potentials for the quadratic transport cost.
Within the variational framework of Kitagawa, Letrouit and M\'erigot,
the first and second variations of the regularized transform reduce
stability to a comparison of two variances. If the source density is bounded above and below on a
bounded John domain, and the target is any probability measure on a fixed
compact set, heat kernel estimates verify
this comparison on both finite-dimensional RCD spaces and Heisenberg
groups. This survey includes two new contributions: an $L^2$
estimate on RCD spaces with the optimal dimension-independent H\"older
exponent $1/2$, and a complete heat kernel proof of the same estimate on
Heisenberg groups. 
\end{abstract}
\maketitle

\section{Introduction}\label{sec:introduction}

This survey concerns the stability of Kantorovich potentials, the functions $\phi$ in the dual problem \eqref{eq:duality}, for the cost
$c(x,y)=\dist^2(x,y)/2$ on a metric measure space $(X,\dist,\mm)$
with $\operatorname{supp}\mm=X$.
We fix a source probability measure $\rho$ and study how the potential
$\phi_\mu$, normalized by $\int\phi_\mu\dd\rho=0$, changes with the
target $\mu$.
Qualitative stability of optimal plans and maps follows from compactness;
see Villani~\cite[Theorem~5.20 and Corollary~5.23]{Villani}.
For potentials, uniform Lipschitz bounds give subsequential uniform
convergence; duality and uniqueness identify the normalized limit.
This compactness argument does not provide a quantitative modulus of continuity.

The quantitative estimate studied here is
\begin{equation}\label{eq:question}
 \norm{\phi_{\mu_1}-\phi_{\mu_0}}_{L^2(\rho)}
 \le C W_1^{1/2}(\mu_0,\mu_1).
\end{equation}
Here $W_1$ is the $1$-Wasserstein distance, the targets lie in a fixed
compact set $Y$, and $C$ is independent of the two targets. On $\PP(Y)$,
$W_1$ metrizes weak convergence~\cite[Theorem~6.9]{Villani}.
Theorems~\ref{thm:rcd} and~\ref{thm:heis} below specify the hypotheses
and the dependence of $C$.

An $\mathrm{RCD}(K,N)$ space, where $K\in\R$ and $1\le N<\infty$,
is a metric measure space with a synthetic Ricci lower bound $K$,
a dimension upper bound $N$, and quadratic Cheeger energy.
The last condition means that the Sobolev space $W^{1,2}(X,\dist,\mm)$
is Hilbertian and the associated heat flow is linear;
see~\cite[Sections~4.3 and~5]{AmbrosioGigliSavare}
and~\cite[Section~3.3]{ErbarKuwadaSturm}.
Examples include complete Riemannian manifolds with the corresponding
Ricci and dimension bounds and their measured Gromov--Hausdorff
limits~\cite[Section~3.3]{ErbarKuwadaSturm}.
Finite-dimensional Alexandrov spaces with curvature bounded below,
equipped with their Hausdorff measure, also belong to this class;
see~\cite{Petrunin} and~\cite[Section~1]{AmbrosioGigliSavare}.

Heat kernel regularization replaces the minimum in the $c$-transform
by a logarithmic integral. The resulting variational argument applies
both to nonsmooth RCD spaces and to the sub-Riemannian Heisenberg groups,
yielding the exponent $1/2$ in \eqref{eq:question}. This exponent is optimal uniformly over dimensions in the RCD class
(Section~\ref{subsec:optimality}).

The argument has four steps. First, replace the $c$-transform by
$\Phi_t[\theta]=-t\log P_{t/2}(e^{\theta/t})$, where $(P_s)_{s>0}$ is the heat
semigroup. Along an affine family of weights, the first variation is
$-u_t$ and the second is $-a_t/t$, where $u_t$ and $a_t$ are a
conditional mean and variance, respectively. Heat kernel estimates then give
\[
 \Var_\rho(u_t)\le\frac Ct\int_S a_t\dd\rho.
\]
Finally, integration in the interpolation parameter and passage to
$t=0$ give \eqref{eq:question}. Sections~\ref{sec:regularization}
and~\ref{sec:variance} explain this reduction; the RCD and Heisenberg
sections give two proofs of the same variance inequality.

\subsection{Previous work and the heat kernel method}

In Euclidean space, quadratic optimal transport is closely related to
convex functions through Brenier's theorem~\cite{Brenier}; the Riemannian
counterpart was developed by McCann~\cite[Theorem~13]{McCann}.
Early quantitative
stability results include Berman~\cite{Berman} and
M\'erigot--Delalande--Chazal~\cite{MerigotDelalandeChazal}.
For a broader survey, see~\cite{LetrouitLecture}.
Delalande--M\'erigot~\cite[Theorem~2.1]{DelalandeMerigot} made quantitative convexity
of the Kantorovich functional central to this topic.
With the sign convention used below, the relevant functional is concave.
Methods for combining local estimates on irregular source domains were
developed by Letrouit and M\'erigot~\cite[Section~3]{LetrouitMerigot}.

Kitagawa--Letrouit--M\'erigot established \eqref{eq:question} on
Riemannian manifolds~\cite[Theorem~1.2]{KitagawaLetrouitMerigot}.
They established the abstract framework used in this survey:
variational identities for the regularized Kantorovich functional,
a sufficient condition for stability in terms of quantitative concavity,
and a method for combining local estimates on John domains in metric
spaces; see~\cite[Sections~2.1, 3.2 and 3.3.1]{KitagawaLetrouitMerigot}.
The use of heat kernel regularization for quantitative stability of
Kantorovich potentials was introduced by the authors~\cite{HanZhuRCD}.
It replaces the kernel $e^{-c(x,y)/t}$ by the heat kernel while
retaining this abstract framework.

The two central parts of this survey verify the variance inequality
in different geometries. The authors~\cite[Theorem~1.1]{HanZhuRCD} proved
an $L^1(\rho)$ estimate with rate $W_1^{1/2}$ on finite-dimensional RCD
spaces. Section~\ref{sec:rcd} obtains a new $L^2(\rho)$ estimate with the
same exponent by summing their local mean oscillation bound over scales
(Lemma~\ref{lem:multiscale}).
Section~\ref{sec:heis} gives a complete new heat kernel proof on $\HH^n$.
These groups fail to satisfy any finite-dimensional $\mathrm{CD}(K,N)$
condition, as discovered by Juillet~\cite{Juillet}.
The main difficulty is to obtain derivative estimates uniform at the
cut locus. We split the heat kernel into two factors at half the time
and integrate over the intermediate point. For distinct endpoints,
the resulting measure concentrates near geodesic midpoints, which
belong to neither endpoint's cut locus~\cite[Lemma~3.3]{NeelSacchelli}.
The resulting integral estimates for logarithmic derivatives are then
used along horizontal segments to prove the variance inequality.

Related work of the authors treats quantitative stability of Wasserstein
barycenters on Alexandrov spaces~\cite{HanZhuBarycenters} and of
multi-marginal optimal transport maps~\cite{HanZhuMultimarginal}.
Here we concentrate on the heat kernel method for source Kantorovich
potentials and \eqref{eq:question}.

\section{Kantorovich duality and potential stability}\label{sec:duality}

\subsection{Kantorovich potentials and normalization}

Let $S\subset X$ be a bounded open source domain and let $Y\subset X$ be
compact. In both settings $X$ is proper, so $\overline S$ is compact.
Write $\PP(Y)$ for the probability measures supported in $Y$ and put
$D=\diam(\overline S\cup Y)$. For a target $\mu\in\PP(Y)$, the dual
formulation of quadratic transport is~\cite[Theorem~5.10]{Villani}
\begin{equation}\label{eq:duality}
 \frac12W_2^2(\rho,\mu)
 =\sup_{\phi(x)+\psi(y)\le c(x,y)}
 \left(\int_S\phi\dd\rho+\int_Y\psi\dd\mu\right).
\end{equation}
A maximizing pair can be chosen $c$-conjugate:
\[
 \phi(x)=\psi^c(x):=\min_{y\in Y}\{c(x,y)-\psi(y)\},
 \qquad
 \psi(y)=\min_{x\in\overline S}\{c(x,y)-\phi(x)\}.
\]
These are called the source and target Kantorovich potentials, respectively. We use the
continuous representatives on $\overline S$ and $Y$. The inequality
$|c(x,y)-c(x,y')|\le D\dist(y,y')$ on $\overline S\times Y\times Y$
shows that $\psi$ is $D$-Lipschitz; by symmetry the same holds for
$\phi$ on $\overline S$.

The pair $(\phi+a,\psi-a)$ has the same dual value. We choose the constant
so that $\E_\rho\phi=0$, where $\E_\rho$ denotes expectation with respect
to $\rho$. Uniqueness of the normalized source potential will also follow
from \eqref{eq:coercivity} by taking $\mu_0=\mu_1$. Quantitative estimates can
also be formulated directly modulo constants by using the variance
$\Var_\rho(f)=\int|f-\E_\rho f|^2\dd\rho$.

\subsection{The Kantorovich functional and a dual estimate}

Define the Kantorovich functional by
\[
 \K[\psi]=\int_S\psi^c\dd\rho.
\]
With our convention $\psi^c(x)=\inf_y\{c(x,y)-\psi(y)\}$, it is
concave. Indeed, for $0\le\lambda\le1$,
\[
 ((1-\lambda)\psi_0+\lambda\psi_1)^c
 \ge(1-\lambda)\psi_0^c+\lambda\psi_1^c. 
\]
Integration against $\rho$ gives the concavity inequality for $\K$.
The convex functional used with Legendre transforms in the Euclidean
formulation~\cite[Section~2]{DelalandeMerigot} has a different sign convention.
For target $\mu$, the dual problem maximizes
$\K[\psi]+\int_Y\psi\dd\mu$. Let $(\phi_i,\psi_i)$ be
$c$-conjugate optimal pairs for $\mu_i$, $i=0,1$. Applying optimality first at
$\mu_0$ and then at $\mu_1$ gives
\[
 -\int_Y(\psi_1-\psi_0)\dd\mu_1
 \le\K[\psi_1]-\K[\psi_0]
 \le-\int_Y(\psi_1-\psi_0)\dd\mu_0.
\]
Thus the pairing
\begin{equation}\label{eq:pairing}
 I(\mu_0,\mu_1)
 :=\int_Y(\psi_1-\psi_0)\dd(\mu_1-\mu_0)
 \ge0
\end{equation}
satisfies, by integration of the bound $\Lip(\psi_1-\psi_0)\le2D$
against a coupling,
\begin{equation}\label{eq:pairing-w1}
 0\le I(\mu_0,\mu_1)\le2D W_1(\mu_0,\mu_1).
\end{equation}
The estimate needed for stability is
\begin{equation}\label{eq:coercivity}
 \norm{\phi_1-\phi_0}_{L^2(\rho)}^2\le C I(\mu_0,\mu_1).
\end{equation}
\subsection{Assumptions on the source measure}\label{subsec:source}

In both settings, the source has density bounded above and below:
\begin{equation}\label{eq:density}
 \dd\rho=f\,\mathbf1_S\dd\mm,
 \qquad 0<a_1\le f\le a_2<\infty\quad\mm\text{-a.e. on }S.
\end{equation}
The domain $S$ is assumed to be a bounded John domain. This means that a base
point $x_*\in S$ and a number $\eta>0$ can be chosen so that every
$x\in S$ can be joined to $x_*$ by an arclength-parametrized curve
$\gamma:[0,\ell]\to S$ satisfying
\begin{equation}\label{eq:john}
 \dist(\gamma(s),X\setminus S)\ge\eta s,
 \qquad 0\le s\le\ell.
\end{equation}
Bounded connected Lipschitz domains in Euclidean space are examples;
see~\cite[Section~1.2]{KitagawaLetrouitMerigot}.

\begin{maintheorem}[RCD stability]\label{thm:rcd}
Let $K\in\R$ and $1\le N<\infty$. On an $\mathrm{RCD}(K,N)$ space
$(X,\dist,\mm)$, let $S\subset X$ be a bounded
$\eta$-John domain and let $\rho=f\mathbf1_S\mm$ be a probability
measure with $0<a_1\le f\le a_2<\infty$ almost everywhere on $S$.
For a nonempty compact set $Y\subset X$, put
$D=\diam(\overline S\cup Y)$. There is a constant
$C=C(K,N,a_1,a_2,\eta,D)$ such that, for every
$\mu_0,\mu_1\in\PP(Y)$, the source Kantorovich potentials for the cost
$c=\dist^2/2$, normalized by $\int_S\phi_{\mu_i}\dd\rho=0$, satisfy
\[
 \norm{\phi_{\mu_1}-\phi_{\mu_0}}_{L^2(\rho)}
 \le C W_1^{1/2}(\mu_0,\mu_1).
\]
In particular, the normalized source potential is unique.
\end{maintheorem}

\begin{maintheorem}[Heisenberg stability]\label{thm:heis}
Let $n\ge1$. Equip $\HH^n$ with the distance and measure specified in
Section~\ref{subsec:heis-heat}: the Carnot--Carath\'eodory distance
$\distcc$ and Lebesgue (Haar) measure $\mm$. Let $S\subset\HH^n$ be a bounded
$\eta$-John domain and let $\rho=f\mathbf1_S\mm$ be a probability
measure with $0<a_1\le f\le a_2<\infty$ almost everywhere on $S$.
For a nonempty compact set $Y\subset\HH^n$, put
$D=\diam(\overline S\cup Y)$. There is a constant
$C=C(n,a_1,a_2,\eta,D)$ such that, for every
$\mu_0,\mu_1\in\PP(Y)$, the source Kantorovich potentials for
$c=\distcc^2/2$, normalized by $\int_S\phi_{\mu_i}\dd\rho=0$, satisfy
\[
 \norm{\phi_{\mu_1}-\phi_{\mu_0}}_{L^2(\rho)}
 \le C W_1^{1/2}(\mu_0,\mu_1).
\]
In particular, the normalized source potential is unique.
\end{maintheorem}

In the common notation used above and in
Sections~\ref{sec:regularization}--\ref{sec:variance},
$\dist=\distcc$ on $\HH^n$; diameters and Wasserstein distances
are computed with respect to this metric.

Both statements allow arbitrary targets in $\PP(Y)$, including atomic
measures. The exponent in Theorem~\ref{thm:rcd} is optimal uniformly
over dimensions; see Section~\ref{subsec:optimality}.
Source assumptions cannot be omitted altogether: certain non-John
domains with thin passages admit no H\"older
modulus~\cite[Theorem~1.9]{LetrouitMerigot}.

\section{Heat kernel regularization of the \texorpdfstring{$c$}{c}-transform}
\label{sec:regularization}

\subsection{The regularized \texorpdfstring{$c$}{c}-transform}

Let $P_s=e^{sL}$ be the heat semigroup and $L$ be its infinitesimal generator, with positive symmetric kernel
$p_s(x,y)$ relative to $\mm$, and assume $P_s1=1$. In Euclidean space
$L=\Delta$. On an RCD space, $L$ is the Laplacian associated with the
Cheeger energy~\cite[Sections~2.2--2.3]{Jiang}. On $\HH^n$, it is the
horizontal sub-Laplacian~\cite[Definition~4.1]{EldredgePrecise}, written
explicitly in Section~\ref{sec:heis} below. For a weight
$\theta:X\to\R$ with the bounds specified in Section~\ref{subsec:extension}, define
\begin{equation}\label{eq:transform}
 \Phi_t[\theta](x)
 =-t\log G_t[\theta](x),
 \qquad G_t[\theta]=P_{t/2}(e^{\theta/t}).
\end{equation}
The heat time is $t/2$ because our cost is $\dist^2/2$. Indeed, in both
settings considered here Varadhan's formula takes the form
\begin{equation}\label{eq:varadhan}
 -t\log p_{t/2}(x,y)\longrightarrow\frac12\dist^2(x,y)
\end{equation}
uniformly on compact sets. In the RCD case this follows from the Gaussian
bounds of Jiang--Li--Zhang~\cite[Theorem~1.2]{JiangLiZhang}, as detailed
in~\cite[Lemma~2.1]{HanZhuRCD}. On Heisenberg groups it follows from
Eldredge's estimates~\cite[Corollary~4.3]{EldredgePrecise}. Thus
\eqref{eq:transform} regularizes the minimum defining $\theta^c$.

The Euclidean formula makes the relationship explicit:
\begin{equation}\label{eq:euclidean-transform}
 \Phi_t[\theta](x)
 =-t\log\int_{\R^m}
 e^{(\theta(y)-|x-y|^2/2)/t}\dd y
 +\frac{mt}{2}\log(2\pi t).
\end{equation}
Subtracting the $\rho$-mean removes the last term. Thus in Euclidean
space the normalized transform agrees with the regularization defined
using the kernel $e^{-|x-y|^2/(2t)}$.
We will also use the heat kernel semigroup formula
\[
 p_{s+r}(x,y)=\int_X p_s(x,z)p_r(z,y)\dd\mm(z),\qquad s,r>0,
\]
which is the kernel form of $P_{s+r}=P_sP_r$.

\subsection{Extension of target potentials}\label{subsec:extension}

A target may be supported on finitely many points, so an integral over $Y$
against $\mm$ could vanish. We therefore apply the heat semigroup on
$X$ to an extension of the $D$-Lipschitz target potential $\psi$.
We use
\begin{equation}\label{eq:extension}
 \begin{split}
 \bar\psi(y)&=\sup_{z\in Y}\{\psi(z)-D\dist(z,y)\},\\
 \psi^*(y)&=\bar\psi(y)-(2D+1)\dist(y,Y).
 \end{split}
\end{equation}
The first line is a Lipschitz extension~\cite{McShane}. The subtracted
function $(2D+1)\dist(y,Y)$ is independent of $\psi$; it ensures that
minimizers of $c(x,\cdot)-\psi^*$ remain in $Y$.

If $z\in Y$ is nearest to $y$ and
$s=\dist(y,Y)$, then for $x\in\overline S$,
$c(x,y)-c(x,z)\ge-Ds$ and $\bar\psi(y)\le\psi(z)+Ds$.
Consequently
\begin{equation}\label{eq:penalty-gap}
 c(x,y)-\psi^*(y)\ge c(x,z)-\psi(z)+s.
\end{equation}
It follows that the infimum over $X$ of $c(x,\cdot)-\psi^*$ is exactly
$\psi^c(x)$, and all minimizers lie in $Y$.
Moreover $\psi^*(y)\le\sup_Y\psi-(D+1)\dist(y,Y)$, so the
function $e^{\psi^*/t}$ decays exponentially in $\dist(y,Y)$ for each $t>0$.

For two target potentials $\psi_0,\psi_1$, put
\begin{equation}\label{eq:segment}
 \theta_i=\psi_i^*,\qquad
 v=\theta_1-\theta_0,\qquad
 \theta_\lambda=\theta_0+\lambda v,\quad0\le\lambda\le1.
\end{equation}
Since the same function is subtracted from both extensions,
$v=\bar\psi_1-\bar\psi_0$. Thus $v$ is bounded and continuous,
$\Lip(v)\le2D$, and
$\norm{v}_\infty\le\norm{\psi_1-\psi_0}_{L^\infty(Y)}$.
Each $\theta_\lambda$ is $(3D+1)$-Lipschitz and satisfies
$\theta_\lambda(y)\le M_0-(D+1)\dist(y,Y)$, where
$M_0=\max\{\sup_Y\psi_0,\sup_Y\psi_1\}$.
We call these weights admissible.

\subsection{Small-time limits}

For $t>0$ and an admissible weight $\theta$, define probability measures
\begin{equation}\label{eq:gibbs}
 \dd\nu_x^{t,\theta}(y)
 =\frac{e^{\theta(y)/t}p_{t/2}(x,y)}{G_t[\theta](x)}\dd\mm(y),
 \qquad
 \nu^{t,\theta}=\int_S\nu_x^{t,\theta}\dd\rho(x).
\end{equation}
The measure $\nu^{t,\theta}$ averages $\nu_x^{t,\theta}$ over starting
points distributed according to $\rho$. The following lemma adapts
\cite[Lemmas~2.12--2.13]{HanZhuRCD} to the extension \eqref{eq:extension}.

\begin{lemma}[Convergence of potentials and measures]
\label{lem:recovery}
In either geometric setting of Theorems~\ref{thm:rcd} and~\ref{thm:heis},
let $\psi:Y\to\R$ be $D$-Lipschitz and let $\psi^*$ be its extension
\eqref{eq:extension}. Then
\begin{equation}\label{eq:laplace}
 \Phi_t[\psi^*]\longrightarrow\psi^c
 \quad\text{uniformly on }\overline S.
\end{equation}
If $c(x,\cdot)-\psi$ has a unique minimizer $T(x)$ on $Y$, then
$\nu_x^{t,\psi^*}\rightharpoonup\delta_{T(x)}$.
In particular, for any $\mu\in\PP(Y)$ and any $c$-conjugate optimal pair
$(\phi_\mu,\psi_\mu)$ in \eqref{eq:duality}, the optimal map $T_\mu$
satisfies
\begin{equation}\label{eq:gibbs-limit}
 \nu_x^{t,\psi_\mu^*}\rightharpoonup\delta_{T_\mu(x)}
 \quad\text{for }\rho\text{-a.e. }x,
 \qquad
 \nu^{t,\psi_\mu^*}\rightharpoonup\mu.
\end{equation}
\end{lemma}

\begin{proof}
Put $F(x,y)=c(x,y)-\psi^*(y)$ and $\phi=\psi^c$.
By \eqref{eq:penalty-gap}, $\min_XF(x,\cdot)=\phi(x)$ and all
minimizers belong to $Y$. Choose a compact neighborhood
$A=\{y:\dist(y,Y)\le R\}$, with $R\ge1$ large enough that
$\sup_{X\setminus A}\psi^*\le-\sup_{\overline S}\phi-1$.
The linear decay of $\psi^*$ permits this choice. Since $P_{t/2}1=1$,
\[
 \int_{X\setminus A}e^{\psi^*(y)/t}p_{t/2}(x,y)\dd\mm(y)
 \le e^{-(\sup_{\overline S}\phi+1)/t}.
\]
By \eqref{eq:varadhan}, the functions
$-t\log p_{t/2}(x,y)-\psi^*(y)$ converge to $F(x,y)$ uniformly on
$\overline S\times A$. Their difference is at most $\varepsilon$
for all sufficiently small $t$.

For each $x$, choose a minimizer $z_x\in Y$. Uniform continuity of
$F$ on $\overline S\times A$ gives $0<\delta<1$, independent of $x$,
such that $F(x,y)\le\phi(x)+\varepsilon$ on $B(z_x,\delta)$.
Full support of $\mm$ and compactness of $Y$ give
$m_\delta:=\inf_{z\in Y}\mm(B(z,\delta))>0$, by a finite covering
with balls of radius $\delta/2$. Consequently
\[
 m_\delta e^{-(\phi(x)+2\varepsilon)/t}
 \le G_t[\psi^*](x)
 \le\bigl(\mm(A)+1\bigr)e^{-(\phi(x)-\varepsilon)/t}.
\]
Taking $-t\log$, then letting $t\downarrow0$ and
$\varepsilon\downarrow0$, proves \eqref{eq:laplace} uniformly.

If $x$ has the unique minimizer $T(x)$, any open neighborhood $U$ of
$T(x)$ has a positive gap
$\inf_{A\setminus U}F(x,\cdot)-\phi(x)>0$, unless
$A\setminus U$ is empty. The same bounds, with $\varepsilon$ smaller
than one quarter of this gap and of $1$, show that
$\nu_x^{t,\psi^*}(X\setminus U)\to0$.
This proves convergence to $\delta_{T(x)}$.

Given an optimal pair $(\phi_\mu,\psi_\mu)$, define
$\phi_\mu=\psi_\mu^c$ on all of $X$, viewing $\psi_\mu$ as $-\infty$
outside $Y$ when taking this $c$-transform. Compactness of $Y$ makes this a
locally Lipschitz $c$-concave function. Its $c$-superdifferential is
single-valued $\mm$-almost everywhere, by
\cite[Theorem~1.3]{GigliRajalaSturm} in the RCD case and
\cite[Theorem~4.4]{AmbrosioRigot} on $\HH^n$.
Every minimizer belongs to this superdifferential; equality in Kantorovich duality
therefore identifies it with the optimal map $T_\mu(x)$, whose
existence follows from~\cite[Theorem~1.1]{GigliRajalaSturm}
and~\cite[Theorem~5.1]{AmbrosioRigot}, respectively.
Since $\rho\ll\mm$, the pointwise limit holds $\rho$-almost everywhere.
Dominated convergence and $(T_\mu)_\#\rho=\mu$ then give the
convergence of $\nu^{t,\psi_\mu^*}$ to $\mu$ in \eqref{eq:gibbs-limit}.
\end{proof}

At positive $t$, $\nu^{t,\psi_\mu^*}$ need not equal $\mu$.
If an atomic measure $\mu_\delta\in\PP(Y)$ is obtained by moving
each point of $\mu$ a distance at most $\delta$, then
$W_1(\mu,\mu_\delta)\le\delta$; \eqref{eq:question} then yields
$\norm{\phi_\mu-\phi_{\mu_\delta}}_{L^2(\rho)}\le C\sqrt\delta$.

\section{Variational identities and the stability estimate}\label{sec:variance}

\subsection{First and second variations}

Average the regularized transform over the source:
\[
 \K_t[\theta]=\int_S\Phi_t[\theta]\dd\rho.
\]
Along the segment \eqref{eq:segment}, write
$F_t(\lambda)=\K_t[\theta_\lambda]$ and set
\[
 u_{t,\lambda}(x)=\E_{\nu_x^{t,\theta_\lambda}}v,
 \qquad
 a_{t,\lambda}(x)=\Var_{\nu_x^{t,\theta_\lambda}}(v).
\]
Differentiating the integral defining $G_t[\theta_\lambda]$ with respect
to $\lambda$, and then differentiating $-t\log G_t[\theta_\lambda]$, yields 
\begin{equation}\label{eq:variations}
 \begin{aligned}
 \partial_\lambda\Phi_t[\theta_\lambda](x)=-u_{t,\lambda}(x),\qquad
 \partial_\lambda^2\Phi_t[\theta_\lambda](x)=-t^{-1}a_{t,\lambda}(x),
 \end{aligned}
\end{equation}
and 
\begin{equation}\label{eq:variation}
	F_t''(\lambda)=-t^{-1}\int_S a_{t,\lambda}\dd\rho. 
\end{equation}
Boundedness of $v$ justifies these differentiations at fixed $t>0$.
The abstract variation formulas are given
in~\cite[Lemmas~2.3--2.4]{KitagawaLetrouitMerigot}; their heat kernel
form is~\cite[Lemma~2.3]{HanZhuRCD}.
Thus $F_t''(\lambda)\le0$, so $F_t$ is concave in $\lambda$.
Since the computation applies to every admissible segment of weights,
$\K_t$ is concave as a functional of $\theta$.

To compare normalized source potentials, set
$\widetilde\Phi_t[\theta]=\Phi_t[\theta]-\K_t[\theta]$. Then
\begin{equation}\label{eq:centered-variation}
 \partial_\lambda\widetilde\Phi_t[\theta_\lambda]
 =-(u_{t,\lambda}-\E_\rho u_{t,\lambda}).
\end{equation}
The variance on the left below measures variation in the source variable
$x$; $a_{t,\lambda}(x)$ measures variation in $y$ for fixed $x$.
The estimate needed to relate them is
\begin{equation}\label{eq:central}
 \Var_\rho(u_{t,\lambda})
 \le\frac Ct\int_S a_{t,\lambda}\dd\rho,
\end{equation}
with $C$ independent of small $t$, of $\lambda$, and of the target measures.

\subsection{A sufficient condition for stability}
\label{subsec:reduction}

The preceding identities give the following form of the variational
reduction of~\cite[Section~3.2]{KitagawaLetrouitMerigot}.

\begin{proposition}[Stability from a variance estimate]
\label{prop:abstract-stability}
Suppose that the regularization of Section~\ref{sec:regularization} and
the limits in Lemma~\ref{lem:recovery} hold.
If \eqref{eq:central} holds with a constant $C_0$ independent of
$t$, $\lambda$ and the two target measures, then their normalized source
potentials satisfy
\[
 \norm{\phi_1-\phi_0}_{L^2(\rho)}^2
 \le C_0 I(\mu_0,\mu_1)\le 2DC_0 W_1(\mu_0,\mu_1).
\]
\end{proposition}

\begin{proof}
Integration of \eqref{eq:centered-variation} and Cauchy--Schwarz in
$\lambda$ give the first inequality below. The second follows from
\eqref{eq:central} and \eqref{eq:variation}: 
\begin{equation}\label{eq:integrated-variation}
 \begin{split}
 \norm{\widetilde\Phi_t[\theta_1]
          -\widetilde\Phi_t[\theta_0]}_{L^2(\rho)}^2
 &\le\int_0^1\Var_\rho(u_{t,\lambda})\dd\lambda
 \le C_0\int_0^1 -F_t''(\lambda)\dd\lambda\\
 &=C_0\left(\E_{\nu^{t,\theta_1}}v
          -\E_{\nu^{t,\theta_0}}v\right).
 \end{split}
\end{equation}
By \eqref{eq:laplace}, the left-hand side converges to
$\norm{\phi_1-\phi_0}_{L^2(\rho)}^2$.
Since $v$ is bounded and continuous and agrees with
$\psi_1-\psi_0$ on $Y$, \eqref{eq:gibbs-limit} identifies the right-hand
limit with $C_0 I(\mu_0,\mu_1)$. This proves
\eqref{eq:coercivity}, and \eqref{eq:pairing-w1} gives
\eqref{eq:question}.
\end{proof}

\subsection{Spatial derivatives and covariance identities}

Fix $\lambda\in[0,1]$ and write $\theta=\theta_\lambda$,
$u_t=u_{t,\lambda}$ and $a_t=a_{t,\lambda}$. All constants below are
independent of $\lambda$. Define the logarithmic gradient
\[
 s_t(x,y)=\nabla_x\log p_{t/2}(x,y).
\]
Differentiating the quotient in \eqref{eq:gibbs} gives the covariance identity
\begin{equation}\label{eq:covariance}
 \begin{split}
 \nabla u_t
 &=\E_{\nuxt}[(v-u_t)(s_t-\E_{\nuxt}s_t)],\\
 |\nabla u_t|^2&\le a_t b_t,
 \qquad b_t=\E_{\nuxt}|s_t-\E_{\nuxt}s_t|^2.
 \end{split}
\end{equation}
In the RCD setting the gradient belongs to the tangent module and the
identity is understood almost everywhere; on $\HH^n$ it is the
horizontal gradient. The weak RCD formulation is established
in~\cite[Lemmas~2.4--2.5]{HanZhuRCD}.

A pointwise bound $b_t\le Ct^{-2}$ is insufficient. Combining it
with \eqref{eq:covariance} and a
Poincar\'e inequality of the form
$\Var_\rho(u_t)\le C_P\int_S|\nabla u_t|^2\dd\rho$ would give
\[
 \Var_\rho(u_t)\le\frac C{t^2}\int_Sa_t\dd\rho
 =\frac Ct\bigl(-F_t''(\lambda)\bigr).
\]
Using this in place of \eqref{eq:central} in the preceding argument
would give only
\[
 \begin{split}
 \norm{\widetilde\Phi_t[\theta_1]-\widetilde\Phi_t[\theta_0]}_{L^2(\rho)}^2\le\frac Ct\left(\E_{\nu^{t,\theta_1}}v-\E_{\nu^{t,\theta_0}}v\right).
 \end{split}
\]
The expression in parentheses converges to $I(\mu_0,\mu_1)$, which need not vanish,
while $C/t$ diverges. Thus this estimate gives no finite bound after
letting $t\downarrow0$. We therefore estimate spatial integrals of
logarithmic heat kernel derivatives to obtain \eqref{eq:central}.

In Euclidean space, if $e$ is a unit
vector, then $e\cdot s_t(x,y)=e\cdot(y-x)/t$, and direct differentiation
gives
\begin{equation}\label{eq:euclidean-score}
 t\Var_{\nuxt}(e\cdot s_t)=1-\partial_e^2\Phi_t[\theta].
\end{equation}
Integrating the second derivative along a segment gives the difference
of first derivatives at its endpoints. A uniform bound on these first
derivatives therefore bounds the integral of
$\Var_{\nuxt}(e\cdot s_t)$ by $C/t$.
On RCD spaces, the corresponding calculation uses the Laplacian of
$\log p_{t/2}$ and integration over balls. On Heisenberg groups,
we estimate horizontal second derivatives of $\log p_{t/4}$ by
integration against the measure defined in \eqref{eq:bridge}.

\subsection{Combining local estimates on John domains}\label{subsec:gluing}

Both geometric arguments lead to the same local estimate. For sufficiently
small balls $B=B(x,r)$ with a fixed enlargement $\Lambda B\subset S$,
\begin{equation}\label{eq:local}
 \int_B|u_t-(u_t)_B|^2\dd\mm
 \le\frac{Cr}{t}\int_{\Lambda B}a_t\dd\mm.
\end{equation}
Here $(u_t)_B$ denotes the $\mm$-average of $u_t$ over $B$, and
$\Lambda B=B(x,\Lambda r)$.
The radius restriction and $\Lambda$ are independent of $t$.

Choose a Whitney covering of $S$ by balls whose radii have a
sufficiently small uniform upper bound, so that \eqref{eq:local} applies
and the enlarged balls have bounded overlap. The John condition gives
chains of overlapping balls joining each ball to a fixed central ball.
The resulting Boman chain condition controls differences of averages
and gives
\[
 \int_S|u_t-(u_t)_S|^2\dd\mm
 \le C\sum_B\int_B|u_t-(u_t)_B|^2\dd\mm
 \le\frac Ct\int_Sa_t\dd\mm.
\]
The bounded radii are absorbed into $C$. The density bounds
\eqref{eq:density} then replace $\mm$ by $\rho$, and the minimizing
property of the mean gives \eqref{eq:central}. The covering construction
and the inequality are~\cite[Proposition~3.7 and Lemma~3.8]{KitagawaLetrouitMerigot},
respectively; the latter extends~\cite[Lemma~3.3]{LetrouitMerigot}.
For the underlying domain theory, see~\cite{BuckleyKoskelaLu,HajlaszKoskela}.

\section{Stability on RCD spaces}\label{sec:rcd}

We verify \eqref{eq:central} in three steps: Li--Yau estimates control
the spatial integral of $b_t$; the covariance identity gives local
$L^1$ mean oscillation; summation over spatial scales gives the required
local $L^2$ estimate.

We use the RCD heat flow introduced in Section~\ref{sec:introduction},
the Gaussian estimates
and local doubling recalled in~\cite[Section~2 and Theorem~1.2]{JiangLiZhang},
and the local Poincar\'e inequality of~\cite[Theorem~1.2]{Rajala}.

\subsection{Li--Yau estimates for the heat kernel}

Fix an admissible weight $\theta$. Write $G_t=P_{t/2}(e^{\theta/t})$.
Heat flow regularization makes $G_t$ and $P_{t/2}(ve^{\theta/t})$
locally Lipschitz, and positivity bounds $G_t$ away from zero on compact
sets. Thus their quotient $u_t$ has the regularity needed in
\eqref{eq:covariance}. The decay of the weights justifies integration
by parts using cutoff functions; see~\cite[Lemmas~2.2, 2.4 and~2.7]{HanZhuRCD}.

Computing the Laplacian of $\log G_t$ gives the identity
\begin{equation}\label{eq:trace}
 b_t=\Delta\log G_t
       -\E_{\nuxt}[\Delta_x\log p_{t/2}(x,\cdot)].
\end{equation}
On an RCD space it holds distributionally~\cite[Lemma~2.6]{HanZhuRCD}.
We estimate its second term using Jiang's Li--Yau
inequality~\cite[Theorem~1.2]{Jiang} in the form
\begin{equation}\label{eq:li-yau}
 |\nabla_x\log p_{t/2}|^2
 \le A_t\frac{\Delta_xp_{t/2}}{p_{t/2}}+B_t,
 \qquad A_t=1+O(t),\quad B_t\le C/t.
\end{equation}
For this purpose one can replace $K$ by $\min\{K,0\}$, so that
$A_t\ge1$. All bounds in this section are for sufficiently small $t$.

Put
\[
 J_t=\E_{\nuxt}|\nabla_x\log p_{t/2}|^2,
 \qquad
 \widetilde J_t=\E_{\nuxt}|\nabla_y\log p_{t/2}|^2.
\]
Symmetry and the heat equation identify $\Delta_xp$ with $\Delta_yp$.
Integration by parts in $y$, against the exponential weight, yields
\begin{equation}\label{eq:ip}
 \left|\E_{\nuxt}\frac{\Delta_xp_{t/2}}{p_{t/2}}\right|
 \le\frac{\Lip(\theta)}t\widetilde J_t^{1/2}.
\end{equation}
Applying \eqref{eq:li-yau} in the $y$ variable first gives
$\widetilde J_t\le Ct^{-2}$, and then applying it in $x$ gives
$J_t\le Ct^{-2}$. Using
$\Delta\log p=\Delta p/p-|\nabla\log p|^2$ in \eqref{eq:li-yau} gives
\begin{equation}\label{eq:trace-bound}
 \begin{split}
 -\E_{\nuxt}\Delta_x\log p_{t/2}
 \le (A_t-1)\E_{\nuxt}\frac{\Delta_xp_{t/2}}{p_{t/2}}+B_t\le C/t.
 \end{split}
\end{equation}
Since $A_t-1=O(t)$, its product with the expectation in
\eqref{eq:ip} is of order $t^{-1}$. Replacing $A_t$ by a fixed constant
greater than one would give only a bound of order $t^{-2}$.

To estimate $b_t$ on a ball, test \eqref{eq:trace} against a cutoff
function equal to one on a ball $B$ of
radius $r\le1$, supported in $2B\subset S$, with gradient bounded by
$C/r$. Since $|\nabla\log G_t|\le J_t^{1/2}\le C/t$, integration by parts
and \eqref{eq:trace-bound} give
\begin{equation}\label{eq:integrated-score}
 \int_B b_t\dd\mm\le\frac{C\mm(2B)}{rt}.
\end{equation}
This is the argument of~\cite[Lemma~2.8]{HanZhuRCD}, with integration
against $\mm$ before applying the upper density bound to pass to $\rho$.

\subsection{A local estimate for mean oscillation}

Combining \eqref{eq:covariance} with \eqref{eq:integrated-score} gives
\[
 \int_B|\nabla u_t|\dd\mm
 \le\left(\frac{C\mm(2B)}{rt}\right)^{1/2}
       \left(\int_Ba_t\dd\mm\right)^{1/2}.
\]
A local weak $(1,1)$-Poincar\'e inequality~\cite[Theorem~1.2]{Rajala}, with a fixed dilation
$\kappa\ge1$, now implies
\begin{equation}\label{eq:mean-osc}
 \frac1{\mm(B)}\int_B|u_t-(u_t)_B|\dd\mm
 \le C\sqrt{\frac rt}
       \left(\frac1{\mm(\kappa B)}
                    \int_{\kappa B}a_t\dd\mm\right)^{1/2},
\end{equation}
whenever $2\kappa B\subset S$ and $r$ is sufficiently small.
Doubling controls the volume ratios introduced by the enlarged balls.
The estimate is uniform in the admissible weight.

This is the local estimate used in~\cite[Section~2.3]{HanZhuRCD}
for $L^1$ stability. For an $L^2$
conclusion, simply interpolating an $L^1$ stability bound with a uniform
$L^\infty$ bound would yield only a $W_1^{1/4}$ rate. The stronger
information available here is that \eqref{eq:mean-osc} holds at every
small spatial scale, with a factor $\sqrt r$.
Summing these bounds over the radii $2^{-j}R$ gives an $L^2$ estimate
with the same exponent $1/2$ of $W_1$.

\subsection{An \texorpdfstring{$L^2$}{L2} estimate by summation over scales}
\label{subsec:multiscale}

\begin{lemma}[$L^2$ estimate from mean oscillation]\label{lem:multiscale}
Suppose that $\mm$ is locally doubling on $S$, with a uniform doubling
constant at the scales below. Let $u\in L^1_{\mathrm{loc}}(S,\mm)$
and $0\le h\in L^1_{\mathrm{loc}}(S,\mm)$. Assume that, for fixed
$C_*>0$, $\kappa\ge1$ and $r_0>0$,
\[
 \frac1{\mm(B')}\int_{B'}|u-u_{B'}|\dd\mm
 \le C_*\sqrt r
       \left(\frac1{\mm(\kappa B')}\int_{\kappa B'}h\dd\mm\right)^{1/2}
\]
for every ball $B'$ of radius $r\le r_0$ with $2\kappa B'\subset S$.
Then, for $B=B(z,R)$ with $4\kappa B\subset S$ and $2R\le r_0$,
\begin{equation}\label{eq:multiscale-estimate}
 \int_B|u-u_B|^2\dd\mm\le CR\int_{2\kappa B}h\dd\mm,
\end{equation}
where $C$ depends only on $C_*$, $\kappa$ and the doubling constant.
\end{lemma}

\begin{proof}
Fix a Lebesgue point $x\in B$ and put $r_j=2^{-j}R$ and
$B_j=B(x,r_j)$. Since $B_{j+1}\subset B_j$ and
$\mm(B_j)\le C_{\mathrm d}\mm(B_{j+1})$ by doubling, the hypothesis gives
\[
 \begin{split}
 |u_{B_{j+1}}-u_{B_j}|
 &\le\frac1{\mm(B_{j+1})}\int_{B_{j+1}}|u-u_{B_j}|\dd\mm\\
 &\le\frac{C_{\mathrm d}}{\mm(B_j)}\int_{B_j}|u-u_{B_j}|\dd\mm\\
 &\le C_{\mathrm d}C_*\sqrt{r_j}
       \left(\frac1{\mm(\kappa B_j)}\int_{\kappa B_j}h\dd\mm\right)^{1/2}.
 \end{split}
\]
For each integer $J\ge1$, the triangle inequality gives
\[
 |u_{B_J}-u_{B_0}|
 \le\sum_{j=0}^{J-1}|u_{B_{j+1}}-u_{B_j}|.
\]
Lebesgue differentiation gives $u_{B_J}\to u(x)$ for almost every
$x\in B$. Letting $J\to\infty$ and using the preceding bound yields
\begin{equation}\label{eq:telescoping}
 \begin{split}
 &|u(x)-u_{B(x,R)}|\\
 &\quad\le C\sum_{j\ge0}\sqrt{r_j}
 \left(\frac1{\mm(B(x,\kappa r_j))}
           \int_{B(x,\kappa r_j)}h\dd\mm\right)^{1/2}.
 \end{split}
\end{equation}
The required balls lie in $S$ because $4\kappa B\subset S$.
Since $B(x,R)\subset2B$ and their volumes are comparable,
$|u_{B(x,R)}-u_{2B}|$ is bounded by a constant times the mean
oscillation on $2B$. The hypothesis at radius $2R$ therefore controls
this initial difference uniformly for $x\in B$.

For any fixed $s\le\kappa R$, Fubini's theorem and doubling give
\begin{equation}\label{eq:fixed-scale}
 \int_B\frac1{\mm(B(x,s))}\int_{B(x,s)}h(y)\dd\mm(y)\dd\mm(x)
 \le C\int_{2\kappa B}h\dd\mm.
\end{equation}
Indeed, if $\dist(x,y)<s$, the volumes of $B(x,s)$ and $B(y,s)$ are
comparable. Integration in $x\in B\cap B(y,s)$ gives a bounded
factor, and every such $y$ lies in $2\kappa B$.
Taking the $L^2(B)$ norm in \eqref{eq:telescoping} and applying
\eqref{eq:fixed-scale} at each scale yields
\[
 \norm{u-u_{B(\cdot,R)}}_{L^2(B,\mm)}
 \le C\sum_{j\ge0}\sqrt{r_j}
             \left(\int_{2\kappa B}h\dd\mm\right)^{1/2}.
\]
The sum of $\sqrt{r_j}$ is $C\sqrt R$.
The $L^2(B)$ norm of $u_{B(\cdot,R)}-u_{2B}$ has the same bound by
the preceding comparison. Since $u_B$ minimizes the squared error
on $B$, this proves \eqref{eq:multiscale-estimate}.
\end{proof}

Apply the lemma with $u=u_t$ and $h=a_t/t$. Equation~\eqref{eq:mean-osc}
then gives
\begin{equation}\label{eq:rcd-local}
 \int_B|u_t-(u_t)_B|^2\dd\mm
 \le\frac{CR}{t}\int_{2\kappa B}a_t\dd\mm.
\end{equation}
Equation~\eqref{eq:rcd-local} gives \eqref{eq:local}, with
$\Lambda=4\kappa$ if needed for the interior condition.
Section~\ref{subsec:gluing} gives \eqref{eq:central}, and
Proposition~\ref{prop:abstract-stability} completes the proof of
Theorem~\ref{thm:rcd}.

\subsection{Optimality of the dimension-independent exponent}
\label{subsec:optimality}

The Euclidean example in~\cite[Remark~3.4]{KitagawaLetrouitMerigot}
shows that no exponent larger than $1/2$ can hold throughout the
finite-dimensional RCD class. Let $\rho$ be uniform on the unit ball
in $\R^d$ and consider the convex Brenier potentials
$q_0(x)=|x|$ and $q_\varepsilon(x)=\max\{|x|,\varepsilon\}$.
The pushforwards $(\nabla q_0)_\#\rho$ and
$(\nabla q_\varepsilon)_\#\rho$ are, respectively, the uniform probability
measure $\mu_0$ on the unit sphere and
$\mu_\varepsilon=(1-\varepsilon^d)\mu_0+\varepsilon^d\delta_0$.
The normalized Kantorovich potentials, obtained by centering
$|x|^2/2-q_i(x)$, satisfy
\[
 W_1(\mu_0,\mu_\varepsilon)=\varepsilon^d,
 \qquad
 \norm{\phi_{\mu_\varepsilon}-\phi_{\mu_0}}_{L^2(\rho)}
 \asymp_d\varepsilon^{(d+2)/2}.
\]
An estimate with exponent $\alpha$ therefore requires
$\alpha\le1/2+1/d$. Allowing $d$ to vary rules out every $\alpha>1/2$,
even when the constant depends on $d$.

\section{Stability on Heisenberg groups}\label{sec:heis}

The proof of Theorem~\ref{thm:heis} requires an integral estimate of
order $t^{-1}$, matching the factor $t^{-1}$ in the second variation
\eqref{eq:variations}. The Euclidean identity \eqref{eq:euclidean-score}
suggests how to obtain it. For an admissible weight $\theta$ and a unit
left-invariant horizontal field $V$, set
$b_{t,V}(x)=\Var_{\nuxt}(V_x\log p_{t/2})$.
We seek bounds of the form
\[
 t b_{t,V}\le C-V^2\Phi_t[\theta],\qquad
 |V\Phi_t[\theta]|\le C.
\]
For a horizontal segment $\gamma$ of bounded length, integration gives
\[
 \int_\gamma b_{t,V}\dd s\le C/t,
\]
because the endpoint values of $V\Phi_t$ are bounded.

The difficulty is to prove these bounds uniformly when the target lies
in the cut locus of the source point. Although $p_{t/2}$ is smooth for
every $t>0$, asymptotic formulas involving second derivatives of
$\distcc^2$ do not apply there. The semigroup formula introduces a
point at time $t/4$. For distinct endpoints, the required derivative
estimates apply near the geodesic midpoints, while the density outside
a fixed neighborhood of these midpoints is exponentially small as
$t\downarrow0$. When $\distcc(x,y)\le\sqrt t/2$, dilation instead
reduces the estimate to a fixed positive time.

\subsection{Geometry and heat kernel estimates}\label{subsec:heis-heat}

Write $\HH^n=\mathbb C^n\times\R$, with multiplication
\[
 [\zeta,\tau][\zeta',\tau']
 =\left[\zeta+\zeta',\tau+\tau'
       +2\operatorname{Im}\sum_j\zeta_j\overline{\zeta'_j}\right].
\]
For $\zeta=\xi+i\eta$, the horizontal fields and sub-Laplacian are
\[
 X_j=\partial_{\xi_j}+2\eta_j\partial_\tau,
 \quad Y_j=\partial_{\eta_j}-2\xi_j\partial_\tau,
 \quad L=\sum_{j=1}^n(X_j^2+Y_j^2).
\]
Let $\distcc$ be the Carnot--Carath\'eodory distance making the fields
$X_j,Y_j$ orthonormal. We use Lebesgue Haar measure $\mm$, which is
invariant under both left and right translations.
All metric balls and Lipschitz constants below refer to $\distcc$.
The dilations $\delta_R[\zeta,\tau]=[R\zeta,R^2\tau]$ multiply distance
by $R$ and volume by $R^Q$, where $Q=2n+2$ is the homogeneous dimension.
The cut locus of the identity
is the punctured central axis; see~\cite[Sections~2--3]{AmbrosioRigot}.
For background on sub-Riemannian transport, see~\cite{Rifford}.
The group is complete and its horizontal distribution is contact, so
nonconstant minimizing geodesics are strongly normal: none of their
nontrivial subsegments is abnormal; see~\cite[Chapter~12]{AgrachevBarilariBoscain}.

Write $p_s(g)=p_s(e,g)$. Left invariance, symmetry and dilation give
\[
 p_s(x,y)=p_s(x^{-1}y)=p_s(y,x),\qquad
 p_s(\delta_Rg)=R^{-Q}p_{s/R^2}(g).
\]
For a horizontal unit vector $\omega$ at the identity, write $V$ and
$\widehat V$ for the left- and right-invariant fields generated by $\omega$.
Their flows through $g$ are $g[h\omega,0]$ and $[h\omega,0]g$.
In particular, differentiation in the starting point of $p_s(x^{-1}z)$
uses $-\widehat V$ in the relative variable $x^{-1}z$.

We use three heat kernel estimates. With $r=\distcc(e,g)$, the bound
of Eldredge~\cite[Corollary~4.3]{EldredgePrecise} implies
\begin{equation}\label{eq:heis-gaussian}
 \frac{C^{-1}s^{-Q/2}e^{-r^2/(4s)}}{(1+r/\sqrt s)^M}
 \le p_s(g)\le Cs^{-Q/2}(1+r/\sqrt s)^M e^{-r^2/(4s)}.
\end{equation}
The positive integer $M$ depends only on $n$ and may be enlarged in
subsequent bounds. Keeping the exact Gaussian exponent is essential when
forming the ratio of three kernels in \eqref{eq:bridge} below.
Second, the logarithmic gradient estimate
\cite[Theorem~4.4]{EldredgePrecise}, dilation and inversion symmetry give
\begin{equation}\label{eq:heis-first-derivative}
 s|\widehat V\log p_s(g)|\le C(\sqrt s+r).
\end{equation}
Third, Theorem~5.10 of Neel--Sacchelli~\cite{NeelSacchelli}, with
derivative order two, gives
$|\widehat V^2\log p_s(g)|\le C_K/s^2$ on compact sets $K$ away
from the identity. Completeness implies the localization condition
in~\cite[Definition~1.1]{NeelSacchelli}; the remaining hypotheses follow
from strong normality of the minimizing geodesics and the bounded
$C^1$ norms of the invariant fields on compact sets.
On compact sets avoiding both the cut locus of $e$ and $e$ itself,
Theorem~5.3 of the same reference gives uniform convergence of
$s\widehat V^2\log p_s$ to
$-\widehat V^2\distcc^2(e,\cdot)/4$, hence the stronger bound $C_K/s$.
This stronger bound requires $K$ to avoid the entire central axis,
including $e$.
Smoothness extends the bounds to
$0<s\le1$, uniformly in the unit vector $\omega$ by expansion in a
fixed horizontal basis. The asymptotics away from the cut locus originate
in~\cite{BenArous}; see also~\cite{BarilariBoscainNeel} at the cut locus.

The estimate on compact sets and dilation give a bound valid on all of $\HH^n$:
\begin{equation}\label{eq:heis-rough-second}
 s|\widehat V^2\log p_s(g)|\le C(1+r^2/s).
\end{equation}
Indeed, set $w=\delta_{s^{-1/2}}g$. The left side is
$|\widehat V^2\log p_1(w)|$, bounded on $\{\distcc(e,w)\le2\}$ by
smoothness and positivity. If $R=\distcc(e,w)>2$, put
$q=\delta_{1/R}w$ and $\varsigma=R^{-2}$. The unit sphere avoids $e$,
so the $C_K/s^2$ bound of~\cite[Theorem~5.10]{NeelSacchelli} applies
there, including at its points on the central axis. Dilation gives
\[
 |\widehat V^2\log p_1(w)|
 =R^{-2}|\widehat V^2\log p_\varsigma(q)|
 \le CR^{-2}\varsigma^{-2}=CR^2.
\]
Thus \eqref{eq:heis-rough-second} also covers the diagonal.

\subsection{Integral estimates for logarithmic derivatives}
\label{subsec:heis-integrals}

The semigroup identity
\[
 p_{2s}(x,y)=\int_{\HH^n}p_s(x,z)p_s(z,y)\dd\mm(z)
\]
leads to the probability measure
\begin{equation}\label{eq:bridge}
 \dd\beta^s_{x,y}(z)
 =\frac{p_s(x,z)p_s(z,y)}{p_{2s}(x,y)}\dd\mm(z).
\end{equation}
It is the distribution at time $s$ of the diffusion with
generator $L$, conditioned to start at $x$ at time $0$ and to end at $y$
at time $2s$. This conditioned process is a diffusion bridge;
see~\cite[Section~6]{NeelSacchelli}.

The relation to geodesic midpoints follows from the Gaussian factors
in \eqref{eq:heis-gaussian}. Their quotient in \eqref{eq:bridge} is
\[
 \exp\left(-\frac{\mathcal E_{x,y}(z)}{4s}\right),\qquad
 \mathcal E_{x,y}(z)
 =\distcc^2(x,z)+\distcc^2(z,y)-\tfrac12\distcc^2(x,y).
\]
The triangle inequality gives $\mathcal E_{x,y}\ge0$, with equality
exactly at midpoints of minimizing geodesics from $x$ to $y$.
For $x\ne y$, such a midpoint $z$ is distinct from both endpoints
and belongs to neither endpoint's cut locus: it lies strictly before
the cut point when the geodesic is followed from either endpoint;
see~\cite[Lemma~3.3]{NeelSacchelli}.
Near these midpoints we use the derivative estimates of
Section~\ref{subsec:heis-heat}. Outside a fixed neighborhood of the
midpoint set, the positive lower bound for $\mathcal E_{x,y}$ gives
exponential decay, which absorbs the negative powers of $s$ in
\eqref{eq:heis-gaussian} and \eqref{eq:heis-rough-second}.

Put $H_t(x,z)=-t\log p_{t/4}(x,z)$.

\begin{proposition}\label{prop:bridge-bounds}
For all $x,y\in\HH^n$, $0<t\le1$ and unit left-invariant horizontal
fields $V$,
\begin{equation}\label{eq:bridge-bounds}
 \begin{aligned}
 \E_{\beta^{t/4}_{x,y}}|V_xH_t(x,Z)|
     &\le C_n(\distcc(x,y)+\sqrt t),\\
 \E_{\beta^{t/4}_{x,y}}[(V_x^2H_t(x,Z))_+]&\le C_n.
 \end{aligned}
\end{equation}
Here $q_+=\max\{q,0\}$. The constants are uniform even when $y$ is
in the cut locus of $x$.
\end{proposition}

\begin{proof}
\emph{Normalization.}
By left translation take $x=e$, write $g=x^{-1}y$, and set
\[
 R=\max\{\distcc(e,g),\sqrt{t}/2\},\qquad
 \sigma=\frac{t}{4R^2},\qquad \bar g=\delta_{1/R}g.
\]
Either $\sigma=1$ and $\distcc(e,\bar g)\le1$, or
$0<\sigma\le1$ and $\distcc(e,\bar g)=1$.
Since $\delta_R$ is a group automorphism, the change of variables
$z=\delta_Ru$ and heat kernel scaling transform the density into
\[
 \frac{R^{-Q}p_\sigma(u)\,R^{-Q}p_\sigma(u,\bar g)}
      {R^{-Q}p_{2\sigma}(\bar g)}R^Q\dd\mm(u)
 =\dd\beta^\sigma_{e,\bar g}(u).
\]
Thus
$(\delta_{1/R})_\#\beta^{t/4}_{e,g}=\beta^\sigma_{e,\bar g}$.
At $z=\delta_R u$, differentiation in the starting point gives
\begin{equation}\label{eq:bridge-derivative-scaling}
 \begin{split}
 (V_xH_t)(e,\delta_R u)&=4R\sigma\widehat V\log p_\sigma(u),\\
 (V_x^2H_t)(e,\delta_R u)&=-4\sigma\widehat V^2\log p_\sigma(u).
 \end{split}
\end{equation}
This follows by differentiating
$-t\log p_{t/4}([-h\omega,0]\delta_R u)$ once or twice at $h=0$.

\emph{Density and moment bounds.}
For $\distcc(e,\bar g)=1$, put $r_1=\distcc(e,u)$,
$r_2=\distcc(u,\bar g)$ and $\mathcal E_{\bar g}(u)=r_1^2+r_2^2-1/2$.
Dividing the two upper bounds in \eqref{eq:heis-gaussian} by the lower
bound at time $2\sigma$ gives
\begin{equation}\label{eq:bridge-energy}
 \frac{\dd\beta^\sigma_{e,\bar g}}{\dd\mm}(u)
 \le C\sigma^{-M}(1+r_1+r_2)^M e^{-\mathcal E_{\bar g}(u)/(4\sigma)}.
\end{equation}
The triangle inequality gives $\mathcal E_{\bar g}(u)\ge0$, with equality precisely at
midpoints of minimizing geodesics from $e$ to $\bar g$. For $r_1\ge2$, $r_2\ge r_1-1$ gives $\mathcal E_{\bar g}(u)\ge r_1^2$.
Since $\mm(B(e,R))=\mm(B(e,1))R^Q$, integration over the annuli
$k\le r_1<k+1$ yields, for a fixed integer $N$,
\begin{equation}\label{eq:bridge-moment}
 \E_{\beta^\sigma_{e,\bar g}}r_1
 \le2+C\sigma^{-M}\sum_{k=2}^\infty(k+1)^N e^{-k^2/(4\sigma)}
 \le C.
\end{equation}
For the last bound, split the exponential into two equal factors;
one absorbs $\sigma^{-M}$ uniformly for $k\ge2$, and the other gives
a summable Gaussian series. If $\sigma=1$ and $\distcc(e,\bar g)\le1$,
positivity of $p_2$ on the unit ball and \eqref{eq:heis-gaussian} give
instead
\[
 \frac{\dd\beta^1_{e,\bar g}}{\dd\mm}(u)
 \le C(1+r_1)^M e^{-r_1^2/4}.
\]
All polynomial moments are then bounded. In either case,
\eqref{eq:heis-first-derivative} and
\eqref{eq:bridge-derivative-scaling} prove the first estimate in
\eqref{eq:bridge-bounds}, since $R\le\distcc(x,y)+\sqrt t/2$.
The same fixed-time density bound and \eqref{eq:heis-rough-second}
prove the second estimate when $\sigma=1$.

\emph{Estimate for the second derivative.}
It remains to treat $\distcc(e,\bar g)=1$ and $0<\sigma<1$;
write $g=\bar g$ in this part of the proof.
The set
\[
 \mathcal M=\{(g,u):\distcc(e,g)=1,\
              \distcc(e,u)=\distcc(u,g)=1/2\}
\]
is compact. For every $(g,u)\in\mathcal M$, the midpoint $u$ differs
from $e$ and does not belong to the cut locus of $e$;
see~\cite[Lemma~3.3]{NeelSacchelli}. Compactness therefore gives
$\varepsilon>0$ such that all points $u$ satisfying
\[
 u\in A_g:=\{r_1\le(1+\varepsilon)/2,\ r_2\le(1+\varepsilon)/2\}
\]
lie in a fixed compact set disjoint from the central axis.
Indeed, otherwise a sequence with $\varepsilon\downarrow0$ would
converge to $\mathcal M$ and meet that axis. The bound
from~\cite[Theorem~5.3]{NeelSacchelli}, extended to positive times by
smoothness as in Section~\ref{subsec:heis-heat}, gives
\[
 \sigma|\widehat V^2\log p_\sigma(u)|\le C
 \quad(u\in A_g,\ 0<\sigma\le1).
\]
Its integral over $A_g$ is bounded because $\beta^\sigma_{e,g}$ is a probability measure.

On $A_g^c$, let $m=\max\{r_1,r_2\}>(1+\varepsilon)/2$.
Since $r_1+r_2\ge1$ and $|r_1-r_2|\le1$, we have
$\min\{r_1,r_2\}\ge\max\{1-m,m-1\}=|1-m|$, so
\[
 \mathcal E_g(u)\ge m^2+(1-m)^2-1/2
   =2(m-1/2)^2\ge\varepsilon^2/2.
\]
Combine this gap with \eqref{eq:heis-rough-second} and
\eqref{eq:bridge-energy}. On $\{r_1<2\}\cap A_g^c$ the integral is
bounded by $C\sigma^{-M-1}e^{-\varepsilon^2/(8\sigma)}$.
On $\{r_1\ge2\}$, the bound $\mathcal E_g(u)\ge r_1^2$ gives
\[
 \begin{split}
 &\int_{A_g^c}\sigma|\widehat V^2\log p_\sigma(u)|
                      \dd\beta^\sigma_{e,g}(u)\\
 &\qquad\le C\sigma^{-M-1}\left(
 e^{-\varepsilon^2/(8\sigma)}
 +\sum_{k=2}^\infty(k+1)^{N+2}e^{-k^2/(4\sigma)}\right)\le C.
 \end{split}
\]
The exponential factors absorb the negative powers of $\sigma$, as in
\eqref{eq:bridge-moment}. Together with the bound on $A_g$ and
\eqref{eq:bridge-derivative-scaling}, this proves the second estimate.
\end{proof}

\subsection{Variance bounds for logarithmic derivatives}
\label{subsec:heis-variance}

Proposition~\ref{prop:bridge-bounds} bounds integrals of derivatives of
$\log p_{t/4}$, whereas $b_{t,V}$ involves $\log p_{t/2}$.
To relate them, fix an admissible weight $\theta=\theta_\lambda$ and define the joint
probability measure of the positions $z$ and $y$ at times $t/4$ and
$t/2$, respectively, by
\begin{equation}\label{eq:joint}
 \dd\Pi_x^{t,\theta}(z,y)
 =\frac{e^{\theta(y)/t}p_{t/4}(x,z)p_{t/4}(z,y)}{G_t[\theta](x)}
       \dd\mm(z)\dd\mm(y).
\end{equation}
Its marginals are $\nuxt$ in $y$ and
\[
 \dd\nuhat(z)
 =\frac{p_{t/4}(x,z)P_{t/4}(e^{\theta/t})(z)}{G_t[\theta](x)}\dd\mm(z).
\]
Conditioned on $y$, the position $z$ at time $t/4$ has distribution
$\beta^{t/4}_{x,y}$.
Differentiating the semigroup formula in $x$ therefore gives
\[
 V_x\log p_{t/2}(x,y)
 =\E_{\beta^{t/4}_{x,y}}[V_x\log p_{t/4}(x,Z)].
\]
Under $\Pi_x^{t,\theta}$, the right-hand side is the conditional
expectation given $y$. Since conditional expectation does not
increase variance,
\begin{equation}\label{eq:contraction}
 \Var_{\nuxt}(V_x\log p_{t/2})
 \le\Var_{\nuhat}(V_x\log p_{t/4}).
\end{equation}
All these differentiations are justified at fixed $t>0$: admissible
weights are bounded above, so $P_{t/4}(e^{\theta/t})$ is bounded; the
first two derivatives of $p_{t/4}$ are bounded by a polynomial times
a Gaussian, by \eqref{eq:heis-gaussian}--\eqref{eq:heis-rough-second}.
This gives locally uniform integrable bounds in $x$ and square
integrability of the logarithmic derivatives under $\nuhat$.

We also need a first moment bound uniform as $t\downarrow0$:
\begin{equation}\label{eq:weighted-first-moment}
 \int\distcc(x,y)\dd\nuxt(y)\le C(n,D).
\end{equation}
Put $L_0=3D+1$, so $\Lip(\theta)\le L_0$.
Integration on $B(x,\sqrt t)$ and heat kernel scaling give
\[
 G_t[\theta](x)\ge c_n e^{\theta(x)/t-L_0/\sqrt t}
                   \ge c_n e^{(\theta(x)-L_0)/t}\quad(0<t\le1).
\]
Since $\theta(y)-\theta(x)\le L_0r$ for $r=\distcc(x,y)$,
\eqref{eq:heis-gaussian} at time $s=t/2$ implies, with a fixed exponent $N_0$,
\[
 \frac{\dd\nuxt}{\dd\mm}(y)
 \le Ct^{-N_0}(1+r)^M
          \exp\left(\frac{L_0r+L_0-r^2/2}{t}\right).
\]
Choose $R_0=R_0(D)\ge2$ so that
$L_0r+L_0-r^2/2\le-r^2/4$ for $r\ge R_0$.
The integral of $r$ inside this ball is at most $R_0$. Outside, the
same annular summation as in \eqref{eq:bridge-moment} bounds
$t^{-N_0}\int_{r\ge R_0}r(1+r)^M e^{-r^2/(4t)}\dd\mm$ uniformly.
This proves \eqref{eq:weighted-first-moment}.

In the density of $\nuhat$, the factor $P_{t/4}(e^{\theta/t})(z)$ is
independent of $x$. Differentiating the logarithm of its normalizing
integral gives
\begin{equation}\label{eq:directional-identity}
 \begin{split}
 V\Phi_t[\theta]&=\E_{\nuhat}V_xH_t,\\
 V^2\Phi_t[\theta]
 &=\E_{\nuhat}V_x^2H_t-t^{-1}\Var_{\nuhat}(V_xH_t).
 \end{split}
\end{equation}
For the second line, the derivative of the probability density is its
product with $-t^{-1}(V_xH_t-\E_{\nuhat}V_xH_t)$.
These identities imply
\begin{equation}\label{eq:directional-bound}
 t\Var_{\nuxt}(V_x\log p_{t/2})
 \le C-V^2\Phi_t[\theta],\qquad |V\Phi_t[\theta]|\le C.
\end{equation}
Indeed, $V_xH_t=-tV_x\log p_{t/4}$, so \eqref{eq:contraction}
and \eqref{eq:directional-identity} give
\[
 \begin{aligned}
 t\Var_{\nuxt}(V_x\log p_{t/2})
 &\le t^{-1}\Var_{\nuhat}(V_xH_t)\\
 &=\E_{\nuhat}V_x^2H_t-V^2\Phi_t[\theta]\\
 &\le\int\E_{\beta^{t/4}_{x,y}}[(V_x^2H_t)_+]\dd\nuxt(y)
          -V^2\Phi_t[\theta]\\
 &\le C_n-V^2\Phi_t[\theta].
 \end{aligned}
\]
The last inequality is \eqref{eq:bridge-bounds}. The bound for
$|V\Phi_t[\theta]|$ follows from the first identity in
\eqref{eq:directional-identity}, by averaging the first estimate in
\eqref{eq:bridge-bounds} against $\nuxt$ and using
\eqref{eq:weighted-first-moment}.
The constants in \eqref{eq:directional-bound} depend only on $n,D$;
in particular they are uniform in $\lambda$ and small $t$.

\subsection{Oscillation along horizontal segments}

Let $\gamma:[0,\ell]\to S$ be a unit-speed horizontal segment in the
constant left-invariant direction $V$. Along this segment $V^2\Phi_t$
is the second derivative of $\Phi_t\circ\gamma$.
With $b_{t,V}$ as defined at the start of this section, integration of
\eqref{eq:directional-bound} gives
\[
 t\int_0^\ell b_{t,V}(\gamma(s))\dd s
 \le C\ell-[(V\Phi_t)(\gamma(s))]_{s=0}^{s=\ell}
 \le C\ell+2C.
\]
The covariance identity \eqref{eq:covariance} gives
$|Vu_t|\le\sqrt{a_t b_{t,V}}$. Cauchy--Schwarz along $\gamma$ then yields
\begin{equation}\label{eq:line}
 \begin{split}
 |u_t(\gamma(\ell))-u_t(\gamma(0))|^2
 &\le\left(\int_\gamma a_t\dd s\right)
       \left(\int_\gamma b_{t,V}\dd s\right)\\
 &\le\frac Ct\int_\gamma a_t\dd s.
 \end{split}
\end{equation}
The constant depends on an upper bound for the segment length.
This argument requires an estimate in each horizontal direction $V$:
a bound for the sub-Laplacian alone does not control the second
derivative along the chosen segment.

\subsection{A local \texorpdfstring{$L^2$}{L2} estimate on balls}

One horizontal segment produces the horizontal displacement, and a
commutator of four segments produces the remaining central displacement.
The following paths allow us to average \eqref{eq:line} over a ball.

\begin{lemma}\label{lem:horizontal-paths}
For $B=B(x_0,r)$ and each $x,y\in B$, there is a Borel choice of a
path $\gamma_{xy}\subset7B$ consisting of at most five horizontal
segments, joining $x$ to $y$, with length at most $6r$. For every
nonnegative Borel function $f$,
\begin{equation}\label{eq:occupation}
 \int_B\int_B\int_{\gamma_{xy}}f\dd s\dd\mm(y)\dd\mm(x)
 \le C_n r\,\mm(B)\int_{7B}f\dd\mm.
\end{equation}
\end{lemma}

\begin{proof}
For $h=x^{-1}y=[\zeta,\tau]\in B(e,2r)$, one has
$|\zeta|\le2r$ and $|\tau|\le4r^2$. For the latter bound, join $e$
to $h$ by a unit-speed horizontal curve of length $\ell<2r$.
Then $|\dot\zeta(s)|=1$ almost everywhere and $|\zeta(s)|\le s$.
The equation $\dot\tau=2\operatorname{Im}\sum_j\zeta_j\overline{\dot\zeta_j}$
gives $|\tau|\le2\int_0^\ell s\dd s=\ell^2$.
Let $e_1$ be the first coordinate vector in $\mathbb C^n$ and put
$\alpha=ir e_1$, $\beta=(\tau/(4r))e_1$. The group law gives
\[
 h=[\zeta,0][\alpha,0][\beta,0][-\alpha,0][-\beta,0].
\]
Indeed, the last four factors have zero horizontal component and
central component $4\operatorname{Im}(\alpha_1\overline{\beta_1})=\tau$.
Following these five factors along horizontal segments defines a path
$\eta_h$ from $e$ to $h$. Its length is at most
$|\zeta|+2r+|\tau|/(2r)\le6r$. The path
$\gamma_{xy}=x\eta_h$ lies in $7B$ and depends Borel measurably on
$(x,y)$, including when a segment has length zero.

For $x\in B$ and $q\in\eta_h$, left invariance and the triangle
inequality give
\[
 \distcc(x_0,xq)\le\distcc(x_0,x)+\distcc(x,xq)
 =\distcc(x_0,x)+\distcc(e,q)<7r.
\]
Thus $Bq\subset7B$. Change variables $y=xh$ and use Fubini.
Right translation $x\mapsto xq$ preserves $\mm$, so
\[
 \begin{split}
 &\int_B\int_B\int_{\gamma_{xy}}f\dd s\dd\mm(y)\dd\mm(x)\\
 =&\int_{B(e,2r)}\int_{\eta_h}
       \int_{B\cap Bh^{-1}}f(xq)\dd\mm(x)\dd s\dd\mm(h)\\
 \leq&\int_{B(e,2r)}\ell(\eta_h)\dd\mm(h)\int_{7B}f\dd\mm\\
 \leq&6r\,\mm(B(e,2r))\int_{7B}f\dd\mm.
 \end{split}
\]
Since $\mm(B(e,2r))=2^Q\mm(B)$, this proves \eqref{eq:occupation}.
\end{proof}

Take a ball with $7B\subset S$. The lengths of the preceding segments
are bounded in terms of $D$. Summing \eqref{eq:line} over at most five
segments and applying Cauchy--Schwarz to this finite sum gives
\[
 |u_t(x)-u_t(y)|^2\le\frac Ct\int_{\gamma_{xy}}a_t\dd s.
\]
The pairwise expression for variance and \eqref{eq:occupation} now give
\begin{equation}\label{eq:heis-local}
 \begin{split}
 \int_B|u_t-(u_t)_B|^2\dd\mm
 &=\frac1{2\mm(B)}\int_B\int_B|u_t(x)-u_t(y)|^2\dd\mm(y)\dd\mm(x)\\
 &\le\frac{Cr}{t}\int_{7B}a_t\dd\mm.
 \end{split}
\end{equation}
This is \eqref{eq:local} with dilation factor $7$. Section~\ref{subsec:gluing}
therefore gives \eqref{eq:central}; Proposition~\ref{prop:abstract-stability}
completes the heat kernel proof of Theorem~\ref{thm:heis}.

\section{Comparison of the heat kernel estimates}

Both arguments obtain the $t^{-1}$ bound by integrating second spatial
derivatives. On RCD spaces, integration by parts against a cutoff
controls the $\Delta\log G_t$ term in \eqref{eq:trace}, while the
Li--Yau inequality controls the other term. On Heisenberg groups,
integration of $V^2\Phi_t$ along a horizontal segment gives first
derivatives at its endpoints. The passage from these estimates to
local $L^2$ bounds is different in the two cases:
\begin{center}
\begin{tabular}{@{}ll@{}}
\hline
RCD spaces & Heisenberg groups\\
\hline
$\Delta_x\log p_{t/2}$ & $V_x^2\log p_{t/4}$\\
Li--Yau inequality & Estimates away from the cut locus\\
Integration over balls & Integration along horizontal segments\\
Summation over radii $2^{-j}R$ & Averaging over horizontal paths\\
\hline
\end{tabular}
\end{center}
An extension to other sub-Riemannian spaces would require uniform
derivative moment bounds as in Proposition~\ref{prop:bridge-bounds}
and a path integral estimate as in Lemma~\ref{lem:horizontal-paths},
together with the analytic and
transport assumptions used in Sections~\ref{sec:regularization}
and~\ref{sec:variance}. Abnormal minimizing curves may prevent the use
of the derivative estimates and the result on geodesic midpoints
in~\cite[Theorems~5.3, 5.10 and Lemma~3.3]{NeelSacchelli}.
In the absence of exact dilations, the normalization to a fixed compact
set also requires a different argument.


\begin{thebibliography}{99}
\raggedright
\interlinepenalty=10000

\bibitem{AgrachevBarilariBoscain}
A. Agrachev, D. Barilari and U. Boscain,
\emph{A Comprehensive Introduction to Sub-Riemannian Geometry},
Cambridge Stud. Adv. Math., vol.~181, Cambridge University Press,
Cambridge, 2020.

\bibitem{AmbrosioGigliSavare}
L. Ambrosio, N. Gigli and G. Savar\'e,
Metric measure spaces with Riemannian Ricci curvature bounded from below,
Duke Math. J. \textbf{163} (2014), 1405--1490.

\bibitem{AmbrosioRigot}
L. Ambrosio and S. Rigot,
Optimal mass transportation in the Heisenberg group,
J. Funct. Anal. \textbf{208} (2004), 261--301.

\bibitem{BarilariBoscainNeel}
D. Barilari, U. Boscain and R. W. Neel,
Small time heat kernel asymptotics at the sub-Riemannian cut locus,
J. Differential Geom. \textbf{92} (2012), 373--416.

\bibitem{BenArous}
G. Ben Arous,
D\'eveloppement asymptotique du noyau de la chaleur hypoelliptique hors du cut-locus,
Ann. Sci. \'Ecole Norm. Sup. (4) \textbf{21} (1988), 307--331.

\bibitem{Berman}
R. J. Berman,
Convergence rates for discretized Monge--Amp\`ere equations and
quantitative stability of optimal transport,
Found. Comput. Math. \textbf{21} (2021), 1099--1140.

\bibitem{Brenier}
Y. Brenier,
Polar factorization and monotone rearrangement of vector-valued functions,
Comm. Pure Appl. Math. \textbf{44} (1991), 375--417.

\bibitem{BuckleyKoskelaLu}
S. M. Buckley, P. Koskela and G. Lu,
Boman equals John,
in \emph{Proceedings of the XVI Rolf Nevanlinna Colloquium},
de Gruyter, Berlin, 1996, pp.~91--99.

\bibitem{DelalandeMerigot}
A. Delalande and Q. M\'erigot,
Quantitative stability of optimal transport maps under variations of the target measure,
Duke Math. J. \textbf{172} (2023), 3321--3357.

\bibitem{EldredgePrecise}
N. Eldredge,
Precise estimates for the subelliptic heat kernel on H-type groups,
J. Math. Pures Appl. (9) \textbf{92} (2009), 52--85.

\bibitem{ErbarKuwadaSturm}
M. Erbar, K. Kuwada and K.-T. Sturm,
On the equivalence of the entropic curvature-dimension condition and
Bochner's inequality on metric measure spaces,
Invent. Math. \textbf{201} (2015), 993--1071.

\bibitem{GigliRajalaSturm}
N. Gigli, T. Rajala and K.-T. Sturm,
Optimal maps and exponentiation on finite-dimensional spaces with Ricci curvature bounded from below,
J. Geom. Anal. \textbf{26} (2016), 2914--2929.

\bibitem{HajlaszKoskela}
P. Haj\l{}asz and P. Koskela,
Sobolev met Poincar\'e,
Mem. Amer. Math. Soc. \textbf{145} (2000), no.~688.

\bibitem{HanZhuRCD}
B.-X. Han and Z.-N. Zhu,
Stability of optimal transport on metric measure spaces,
preprint, \href{https://arxiv.org/abs/2602.19175v2}{arXiv:2602.19175v2}, 2026.

\bibitem{HanZhuBarycenters}
B.-X. Han and Z.-N. Zhu,
Quantitative stability of Wasserstein barycenters over Alexandrov spaces
with lower curvature bounds,
preprint, \href{https://arxiv.org/abs/2605.25448v1}{arXiv:2605.25448v1}, 2026.

\bibitem{HanZhuMultimarginal}
B.-X. Han and Z.-N. Zhu,
Two-mode stability for multi-marginal optimal transport maps,
preprint, \href{https://arxiv.org/abs/2606.23037v1}{arXiv:2606.23037v1}, 2026.

\bibitem{Jiang}
R. Jiang,
The Li--Yau inequality and heat kernels on metric measure spaces,
J. Math. Pures Appl. (9) \textbf{104} (2015), 29--57.

\bibitem{JiangLiZhang}
R. Jiang, H. Li and H. Zhang,
Heat kernel bounds on metric measure spaces and some applications,
Potential Anal. \textbf{44} (2016), 601--627.

\bibitem{Juillet}
N. Juillet,
Geometric inequalities and generalized Ricci bounds in the Heisenberg group,
Int. Math. Res. Not. IMRN \textbf{2009} (2009), 2347--2373.

\bibitem{KitagawaLetrouitMerigot}
J. Kitagawa, C. Letrouit and Q. M\'erigot,
Stability of optimal transport maps on Riemannian manifolds,
preprint, \href{https://arxiv.org/abs/2504.05412v2}{arXiv:2504.05412v2}, 2025.

\bibitem{LetrouitMerigot}
C. Letrouit and Q. M\'erigot,
Gluing methods for quantitative stability of optimal transport maps,
to appear in Ann. Sci. \'Ecole Norm. Sup.; \href{https://arxiv.org/abs/2411.04908v3}{arXiv:2411.04908v3}, 2026.

\bibitem{LetrouitLecture}
C. Letrouit,
\href{https://www.imo.universite-paris-saclay.fr/~cyril.letrouit/teaching/Peccotfinal.pdf}
{Lectures on quantitative stability of optimal transport},
Cours Peccot lecture notes, Coll\`ege de France, 2025.

\bibitem{McCann}
R. J. McCann,
Polar factorization of maps on Riemannian manifolds,
Geom. Funct. Anal. \textbf{11} (2001), 589--608.

\bibitem{McShane}
E. J. McShane,
Extension of range of functions,
Bull. Amer. Math. Soc. \textbf{40} (1934), 837--842.

\bibitem{MerigotDelalandeChazal}
Q. M\'erigot, A. Delalande and F. Chazal,
Quantitative stability of optimal transport maps and linearization of the
$2$-Wasserstein space,
Proc. Mach. Learn. Res. \textbf{108} (2020), 3186--3196.

\bibitem{NeelSacchelli}
R. W. Neel and L. Sacchelli,
Uniform, localized asymptotics for sub-Riemannian heat kernels, their logarithmic derivatives, and associated diffusion bridges,
J. Theoret. Probab. \textbf{39} (2026), no.~3, Paper No.~57, 103~pp.
Numbered results cited here refer to
\href{https://arxiv.org/abs/2012.12888v2}{arXiv:2012.12888v2}.

\bibitem{Petrunin}
A. Petrunin,
Alexandrov meets Lott--Villani--Sturm,
M\"unster J. Math. \textbf{4} (2011), 53--64.

\bibitem{Rajala}
T. Rajala,
Local Poincar\'e inequalities from stable curvature conditions on metric spaces,
Calc. Var. Partial Differential Equations \textbf{44} (2012), 477--494.

\bibitem{Rifford}
L. Rifford,
\emph{Sub-Riemannian Geometry and Optimal Transport},
SpringerBriefs in Mathematics, Springer, Cham, 2014.

\bibitem{Villani}
C. Villani,
\emph{Optimal Transport: Old and New},
Grundlehren Math. Wiss., vol.~338, Springer, Berlin, 2009.

\end{thebibliography}
\end{document}